\documentclass[reqno,12pt]{amsart}
\usepackage{latexsym,amssymb,amscd}

\def\B'c{{\mathcal{B'}}}
\def\U'c{{\mathcal{U'}}}

\def\opn#1#2{\def#1{\operatorname{#2}}} 
\opn\chara{char}
\opn\length{\ell}
\opn\projdim{proj\,dim}
\opn\injdim{inj\,dim}
\opn\ini{in}
\opn\rank{rank}
\opn\depth{depth}
\opn\height{ht}
\opn\embdim{emb\,dim}
\opn\codim{codim}

\opn\Tr{Tr}
\opn\bigrank{big\,rank}
\opn\superheight{superheight}\opn\lcm{lcm}
\opn\trdeg{tr\,deg}%
\opn\reg{reg}
\opn\lreg{lreg}
\opn\set{set}
\opn\supp{Supp}
\opn\shad{Shad}
\opn\div{div}
\opn\Div{Div}
\opn\cl{cl}
\opn\Cl{Cl}
\opn\Spec{Spec}
\opn\Supp{Supp}
\opn\supp{supp}
\opn\Sing{Sing}
\opn\Ass{Ass}
\opn\Ann{Ann}
\opn\Rad{Rad}
\opn\Soc{Soc}
\opn\Ker{Ker}
\opn\Coker{Coker}
\opn\Im{Im}
\opn\Hom{Hom}
\opn\Tor{Tor}
\opn\Ext{Ext}
\opn\End{End}
\opn\Aut{Aut}
\opn\id{id}

\opn\nat{nat}
\opn\GL{GL}
\opn\SL{SL}
\opn\mod{mod}
\opn\ord{ord}
\opn\aff{aff}
\opn\con{conv}
\opn\relint{relint}
\opn\st{st}
\opn\lk{lk}
\opn\cn{cn}
\opn\core{core}
\opn\vol{vol}
\opn\gr{gr}

\def\pot#1#2{#1[\kern-0.28ex[#2]\kern-0.28ex]}

\opn\dirlim{\underrightarrow{\lim}}
\opn\invlim{\underleftarrow{\lim}}
\def\pnt{{\raise0.5mm\hbox{\large\bf.}}}

\def\Implies{\ifmmode\Longrightarrow \else
     \unskip${}\Longrightarrow{}$\ignorespaces\fi}
\def\implies{\ifmmode\Rightarrow \else
     \unskip${}\Rightarrow{}$\ignorespaces\fi}
\def\iff{\ifmmode\Longleftrightarrow \else
     \unskip${}\Longleftrightarrow{}$\ignorespaces\fi}

\let\:=\colon
\newtheorem{Theorem}{Theorem}[section]
\newtheorem{Lemma}[Theorem]{Lemma}

\newtheorem{Example}[Theorem]{Example}

\newtheorem{Definition}[Theorem]{Definition}

\let\epsilon=\varepsilon
\let\phi=\varphi
\let\kappa=\varkappa
\title{Gotzmann lexsegment ideals}
\author{Anda Olteanu \and Oana Olteanu \and Loredana Sorrenti}

\thanks{The third author was partially supported by Regional Research Grant A1UNIRC017 from Calabria (2008).}
\address{Faculty of Mathematics and Computer Science, Ovidius University, Bd.\ Mamaia 124,
 900527 Constanta, Romania,} \email{olteanuandageorgiana@gmail.com} 
\address{Faculty of Mathematics and Computer Science, Ovidius University, Bd.\ Mamaia 124,
 900527 Constanta, Romania,} \email{olteanuoanastefania@yahoo.com} 
\address{DIMET University of Reggio Calabria, Faculty of Engineering, via Graziella (Feo di Vito), 89100 Reggio Calabria, Italy} \email{loredana.sorrenti@unirc.it.}

\begin{document}

\begin{abstract} In this paper we characterize the componentwise lexsegment ideals which are componentwise linear and the lexsegment ideals generated in one degree which are Gotzmann.

\end{abstract}
\maketitle
\section{Introduction}
Let $k$ be a field and $S=k[x_1,\ldots,x_n]$ the ring of polynomials in $n$ variables. We consider $S$ to be standard graded, that is $\deg(x_i)=1$ for all $i$. We denote by $\frak m$ the maximal graded ideal of $S$. Let $I\subset S$ be a graded ideal, $I=\bigoplus\limits_{q\geq 0}I_q$. We denote by $H(I,-)$ its Hilbert function, that is $H(I,q)=\dim_k(I_q)$ for all $q\geq 0$, and by $I_{\langle q\rangle}$ the homogeneous ideal generated by the component of degree $q$ of $I$. 

In \cite{HH}, J. Herzog and T. Hibi defined the componentwise linear ideals. Namely, a graded ideal $I$ of $S$ is called \textit{componentwise linear} if, for each degree $q$, $I_{\langle q\rangle}$ has a linear resolution. 

A. Soleyman Jahan and X. Zheng generalized the notion of ideal with linear quotients as follows: a graded ideal $I$ \textit{has componentwise linear quotients} if, for each degree $q$, $I_{\langle q\rangle}$ has linear quotients. They proved that any graded ideal with linear quotients has componentwise linear quotients \cite[Theorem 2.7]{SZ}.

Since any ideal with linear quotients generated in one degree has a linear resolution \cite{CH}, looking at the above definitions, one may note that any graded ideal with componentwise linear quotients is componentwise linear. In general, the converse does not hold.

Along the above definitions, one may define componentwise lexsegment ideals (see Definition \ref{compwiselex}). We prove that, for this class of ideals, being componentwise linear is equivalent to having componentwise linear quotients.

Let $d$ be a positive integer. Then any non-negative integer $a$ has a unique representation of the form 
$$a={a_d\choose d}+\ldots +{a_j\choose j},$$ 
where $a_d>a_{d-1}>\ldots>a_j\geq j\geq1$. This is called the \textit{binomial} or \textit{Macaulay expansion of $a$ with respect to $d$}. For such an expansion of $a$ with respect to $d$ one defines  
$$a^{\langle d\rangle}={a_d+1\choose d+1}+\ldots +{a_j+1\choose j+1}.$$

It is customary to put $0^{\langle d\rangle}=0$ for any $d>0$.

We recall the Gotzmann's persistence theorem \cite{G}.

\begin{Theorem}Let $I\subset S$ be a homogeneous ideal generated by elements of degree at most $d$. If $H(I,d+1)=H(I,d)^{\langle d\rangle}$, then $H(I,q+1)=H(I,q)^{\langle q\rangle}$ for all $q\geq d$.   
\end{Theorem}

Given a graded ideal $I\subset S$, there exists a unique lexicographic ideal $I^{lex}$ such that $I$ and $I^{lex}$ have the same Hilbert function. The lexicographic ideal $I^{lex}$ is constructed as follows. For each graded component $I_j$ of $I$, one consider $SI_j^{lex}$ to be the ideal generated by the unique initial lexsegment $\mathcal{L}_j$ such that $|\mathcal{L}_j|=\dim_k(I_j)$. Let $I^{lex}=\bigoplus\limits_{j} I_j^{lex}$. It is known that $I^{lex}$ constructed as before is indeed an ideal.

A graded ideal $I\subset S$ generated in degree $d$ is called a \textit{Gotzmann ideal} if the number of generators of $\frak mI$ is the smallest as possible, namely it is equal to the number of generators of $\frak mI^{lex}$. Therefore, by Gotzmann's persistence theorem, a graded ideal $I\subset S$ generated in degree $d$ is Gotzmann if and only if $I$ and $(I^{lex})_{\langle d\rangle}$ have the same Hilbert function. 

For graded Gotzmann ideals we have the following characterization in terms of (graded) Betti numbers \cite{HH}.

\begin{Theorem}Let $I\subset S$ be a graded ideal. The following conditions are equivalent:
\begin{itemize}
	\item [(a)] $\beta_{ij}(S/I)=\beta_{ij}(S/I^{lex})$ for all $i,j$;
	\item [(b)] $\beta_{1j}(S/I)=\beta_{1j}(S/I^{lex})$ for all $j$; 
	\item [(c)] $\beta_1(S/I)=\beta_1(S/I^{lex})$;
	\item [(d)] $I$ is a Gotzmann ideal.
\end{itemize}     
\end{Theorem}

Let $I$ be a Gotzmann monomial ideal generated in degree $d$. From the above results it follows that $I^{lex}$ is also generated in degree $d$ and $I$ has a linear resolution.

We aim at characterizing the lexsegment ideals generated in one degree which are Gotzmann.

For an integer $d\geq 2$, let $\mathcal M_d$ be the set of all monomials of degree $d$ in $S$ ordered lexicographically with $x_1>x_2>\ldots>x_n$. A lexsegment ideal of $S$ generated in degree $d$ is a monomial ideal generated by a segment of $\mathcal M_d$, that is by a set of monomials of the form
$$\mathcal L(u,v)=\{w\in\mathcal M_d\ |\ u\geq w\geq v\},$$  
where $u,\ v\in\mathcal M_d$, $u\geq v$.

A monomial ideal generated by an initial lexsegment $\mathcal L^i(v)=\{w\in\mathcal M_d\ |\ w\geq v\}$, $v\in\mathcal M_d$, is called an \textit{initial lexsegment ideal}.
 
Initial lexsegment ideals are obviously Gotzmann.

Arbitrary lexsegment ideals with linear resolutions have been characterized in \cite{ADH}. Their characterization distinguishes between completely and non-completely lexsegment ideals. In order to characterize the Gotzmann property of a lexsegment ideal we also need to distinguish between these two classes of ideals. In the last two sections of this paper, we analyze these two classes.

Our paper gives a complete solution to a problem posed by Professor J. Herzog at the School of Research PRAGMATIC 2008 in Catania, July 2008.
	
	\[
\]
\textbf{Acknowledgment.} The first and the third authors are grateful to the organizers of the School of Research PRAGMATIC 2008, Catania, Italy. The authors would like to thank Professor J\"urgen Herzog and Professor Viviana Ene for valuable discussions and comments during the preparation of this paper.  

\section{Componentwise lexsegment ideals}

We define the componentwise lexsegment ideals and we characterize all the componentwise lexsegment ideals which are componentwise linear.

\begin{Definition}\label{compwiselex}\rm Let $I$ be a monomial ideal in $S$ and $d$ the least degree of the minimal monomial generators. The ideal $I$ is called \textit{componentwise lexsegment} if, for all $j\geq d$, its degree $j$ component $I_j$ is generated, as $k$-vector space, by the lexsegment set $\mathcal{L}(x_1^{j-d}u,vx_n^{j-d})$.
\end{Definition}

Obviously, completely lexsegment ideals are componentwise lexsegment ideals as well.

\begin{Example}\rm  The ideal $I=(x_1x_3^2,x_2^3,x_1x_2^2x_3)$ is a componentwise lexsegment ideal. Indeed, one may note that $I_3$ is the $k$-vector space spanned by $\mathcal{L}(x_1x_3^2,x_2^3)$ and $I_4$ is the $k$-vector space generated by $\mathcal{L}(x_1^2x_3^2,x_2^3x_3)$. Since $\mathcal{L}(x_1^2x_3^2,x_2^3x_3)$ is a completely lexsegment set \cite[Theorem 2.3]{DH}, $I_j$ is generated by the lexsegment set $\mathcal{L}(x_1^{j-2}x_3^2,x_2^3x_3^{j-3})$ for all $j\geq 4$.
\end{Example}

We characterize all the componentwise lexsegment ideals which are componentwise linear. In the same time, we prove the equivalence of the notions componentwise linear ideal and componentwise linear quotients for this particular class of graded ideals.
 
One may note that we can assume $x_1\mid u$ since otherwise we can study the ideal in a polynomial ring in a smaller number of variables.
 
\begin{Theorem} Let $I$ be a componentwise lexsegment ideal and $d\geq1$ the lowest degree of the minimal monomial generators of $I$. Let $u,v\in\mathcal{M}_d$, $x_1|u$ be such that $I_{\langle d\rangle}=(\mathcal{L}(u,v))$. The following conditions are equivalent:

\begin{itemize}
	\item[(a)] $I$ is a componentwise linear ideal.
	\item[(b)] $I_{\langle d\rangle}$ has a linear resolution.
	\item[(c)] $I_{\langle d\rangle}$ has linear quotients.
	\item[(d)] $I$ has componentwise linear quotients.
\end{itemize}
\end{Theorem}

\begin{proof} (a)$\Rightarrow$(b) Since $I$ is componentwise linear, the statement is straightforward.

(b)$\Rightarrow$(c) This was proved in \cite[Theorem 1.1 and Theorem 2.1]{EOS} by analyzing separately the cases of completely and non-completely lexsegment ideals.

(c)$\Rightarrow$(d) We separately treat the case of completely and of non-completely lexsegment ideals. Firstly, let us assume that $I_{\langle d\rangle}$ is a completely lexsegment ideal with linear quotients. Hence $I=I_{\langle d\rangle}$ and $I$ has componentwise linear quotients \cite[Theorem 2.7]{SZ}.

If $I_{\langle d\rangle}=(\mathcal{L}(u,v))$ is a non-completely lexsegment ideal with linear quotients, then $I_{\langle d\rangle}$ has a linear resolution and, by \cite[Theorem 2.4]{ADH}, $u$ and $v$ must have the form
	\[u=x_1x_{l+1}^{a_{l+1}}\cdots x_n^{a_n}\ \mbox{and}\ v=x_lx_n^{d-1}
\]
for some $l$, $2\leq l<n$. Therefore, $\nu_1(u)=1$ and $\nu_1(v)=0$. Here, for a monomial $m=x_1^{a_1}\cdots x_n^{a_n}$, we denoted by $\nu_i(m)$ the exponent of the variable $x_i$, that is $\nu_i(m)=a_i$.

If we look at the ends of the lexsegment $\mathcal{L}(x_1u,vx_n)$, we have $\nu_1(x_1u)=2$, $\nu_1(vx_n)=0$ and one may easily see that $(\mathcal{L}(x_1u,vx_n))$ is a completely lexsegment ideal. By \cite[Theorem 1.3]{ADH}, $(\mathcal{L}(x_1u,vx_n))$ has a linear resolution and, using \cite[Theorem 2.1]{EOS}, $(\mathcal{L}(x_1u,vx_n))$ has linear quotients. Since $(\mathcal{L}(x_1u,vx_n))$ is a completely lexsegment ideal with a linear resolution, the ideals generated by the shadows of $\mathcal{L}(x_1u,vx_n)$ are completely lexsegment ideals with linear resolutions, hence they have linear quotients by \cite[Theorem 2.1]{EOS}. Therefore, $I$ has componentwise linear quotients.  

(d)$\Rightarrow$(a) Since any ideal with linear quotients generated in one degree has a linear resolution, the statement follows by comparing the definitions.
\end{proof}

\section{Gotzmann completely lexsegment ideals}

In this section we are going to characterize the completely lexsegment ideals gene- rated in degree $d$ which are Gotzmann.

Firstly we recall another operator connected with the binomial expansion of an integer.

Let $a={a_d\choose d}+\ldots +{a_j\choose j}$, $a_d>a_{d-1}>\ldots>a_j\geq j\geq1$, be the binomial expansion of $a$ with respect to $d$. Then
$$a^{(d)}={a_d\choose d+1}+\ldots +{a_j\choose j+1}.$$

We obviously have the following equality:
$$a^{\langle d\rangle}=a+a^{(d)}.$$

\begin{Lemma}\label{b(d)=c(d)} Let $c>b>0$ be two integers. Let
	$b={b_d\choose d}+\ldots +{b_j\choose j},$ $b_d>b_{d-1}>\ldots>b_j\geq j\geq1$, and $c={c_d\choose d}+\ldots +{c_i\choose i},$ $c_d>c_{d-1}>\ldots>c_i\geq i\geq1$, be the $d$-binomial expansions of $b$ and $c$. The following statements are equivalent:
\begin{itemize}
	\item[(i)] $b^{(d)}=c^{(d)}$;
	\item[(ii)] $j\geq2$ and $c-b\leq j-1$. 
\end{itemize}
\end{Lemma}

\begin{proof} Let $b^{(d)}=c^{(d)}$. Since $c>b$, by \cite[Lemma 4.2.7]{BH}, there exists $s\leq d$ such that $c_d=b_d,\ldots,c_{s+1}=b_{s+1}$, and $c_s>b_s$. We obviously have $s+1\geq j$. Let us suppose that $s\geq j$. Since $c_s\geq b_s+1$, we get:
	\[{c_s\choose s+1}\geq{b_s+1\choose s+1}\geq{b_s\choose s+1}+{b_{s-1}\choose s}+\ldots+{b_j\choose j+1}+{b_j\choose j}>\]
	\[>{b_s\choose s+1}+{b_{s-1}\choose s}+\ldots+{b_j\choose j+1}
\]
This leads to the inequality $c^{(d)}>b^{(d)}$, which contradicts our hypothesis. Indeed, we have
	\[c^{(d)}\geq{c_d\choose d+1}+\ldots+{c_{s+1}\choose s+2}+{c_s\choose s+1}>\]
  \[>{b_d\choose d+1}+\ldots+{b_{s+1}\choose s+2}+{b_s\choose s+1}+\ldots+{b_j\choose j+1}=b^{(d)}.\]
Therefore we must have $s=j-1$. Hence $j\geq2$ and $c$ has the binomial expansion
	\[c={c_d\choose d}+\ldots+{c_j\choose j}+{c_{j-1}\choose j-1}+\ldots+{c_i\choose i}.
\]
Using the equality $c^{(d)}=b^{(d)}$ we get
	\[{c_{j-1}\choose j}+\ldots+{c_i\choose i+1}=0,
\]
which implies that $c_{j-1}=j-1,\ldots,c_i=i$. Therefore $c=b+j-i\leq b+j-1$, which proves (ii).
 
Now, let $j\geq2$ and $c\leq b+j-1$. As in the first part of the proof, let $s\leq d$ be an integer such that $c_d=b_d,\ldots,c_{s+1}=b_{s+1}$, and $c_s>b_s$. If $s\geq j$, we get the following inequalities:
	\[c={c_d\choose d}+\ldots+{c_{s+1}\choose s+1}+{c_s\choose s}+\ldots+{c_i\choose i}\geq\]
	\[\geq{b_d\choose d}+\ldots+{b_{s+1}\choose s+1}+{b_s+1\choose s}+{c_{s-1}\choose s-1}+\ldots+{c_i\choose i}\geq\]
	\[\geq{b_d\choose d}+\ldots+{b_{s+1}\choose s+1}+{b_{s}\choose s}+\ldots+{b_j\choose j}+{b_j\choose j-1}+{c_{s-1}\choose s-1}+\ldots+{c_i\choose i}=
\]
	\[=b+{b_j\choose j-1}+{c_{s-1}\choose s-1}+\ldots+{c_i\choose i}\geq b+j-i+s.
\]
Since, by hypothesis, $c-b\leq j-1$, we have $j-1\geq j-i+s$, thus $s\leq i-1$, a contradiction. Hence, $s=j-1$. Then we have:

	\[c^{(d)}={c_d\choose d+1}+\ldots+{c_{s+1}\choose s+2}+{c_s\choose s+1}+\ldots+{c_{i}\choose i+1}=
\]
	\[={b_d\choose d+1}+\ldots+{b_{j}\choose j+1}+{c_{j-1}\choose s}+\ldots+{c_{i}\choose i+1}=\]\[=b^{(d)}+{c_{j-1}\choose j}+\ldots+{c_{i}\choose i+1}
\]
If we assume that $c_{j-1}\geq j$, then it follows that ${c_{j-1}\choose j-1}\geq j$. Looking at the $d$-binomial expansions of $b$ and $c$, we get $c-b\geq j$, contradiction. Hence $c_{j-1}=j-1$. This equality implies also the equalities $c_k=k$, for all $i\leq k\leq j-2$. We obtain the following binomial expansion of $c$:
	\[c={c_d\choose d}+\ldots+{c_{j}\choose j}+{j-1\choose j-1}+\ldots+{i\choose i}.
\]
Then
\[c^{(d)}={c_d\choose d+1}+\ldots+{c_{j}\choose j+1}={b_d\choose d+1}+\ldots+{b_{j}\choose j+1}=b^{(d)}.
\] 
\end{proof}

\begin{Lemma}\label{c^d=0 dnd c<egal b}
Let $c>0$ be an integer with the binomial expansion
\[c={c_d\choose d}+\ldots+{c_i\choose i},c_d>\ldots>c_i\geq i\geq 1.
\]
The following statements are equivalent:
\begin{itemize}
	\item [(a)] $c^{(d)}=0$; 
	\item [(b)] $c\leq d$. 
\end{itemize}
\end{Lemma}

\begin{proof}
Let $c\leq d$. Then $c$ has the following binomial expansion with respect to $d$:
\[c={d\choose d}+\ldots+{i\choose i}, \mbox{ for some }i\geq 1.
\]
Hence $c^{(d)}=0$.

Now let $c^{(d)}=0$. We get
\[{c_d\choose d+1}+\ldots+{c_i\choose i+1}=0,\mbox{ which implies}\]
$$c_d=d,\ldots,\ c_i=i.$$

It follows $c=d-(i-1)\leq d$.
\end{proof}

\begin{Theorem}\label{compG} Let $u,v\in \mathcal M_d$, $x_1\mid u$ such that $I=(\mathcal{L}(u,v))$ is a completely lexsegment ideal of $S$ which is not an initial lexsegment ideal. Let $j$ be the exponent of the variable $x_n$ in $v$ and $a=|\mathcal M_d\setminus\mathcal{L}^i(u)|$. The following statements are equivalent:
\begin{itemize}
	\item [(a)] $I$ is a Gotzmann ideal;
	\item [(b)] $a\geq{n+d-1\choose d}-(j+1)$.
\end{itemize}
\end{Theorem}

\begin{proof} Let $b=|\mathcal M_d\setminus\mathcal{L}^i(v)|$ and $w\in \mathcal M_d$ such that $|\mathcal{L}(u,v)|=|\mathcal{L}^i(w)|$. We denote $c=|\mathcal M_d\setminus\mathcal{L}^i(w)|$. Then $|\mathcal{L}^i(w)|=|\mathcal{L}^i(v)|-|\mathcal{L}^i(u)|+1=a-b+1$, which yields:
	\[{n+d-1\choose d}-c=a-b+1,
\]
that is
	\begin{eqnarray}c={n+d-1 \choose d}-(a+1)+b. \label{1}
\end{eqnarray}
Since $I$ is completely, $I$ is Gotzmann if and only if 
	\begin{eqnarray}|\mathcal{L}^i(wx_n)|=|\mathcal{L}(ux_1,vx_n)|=|\mathcal{L}^i(vx_n)|-|\mathcal{L}^i(ux_1)|+1.\label{2}
\end{eqnarray}
Since $x_1\mid u$, we have $|\mathcal{L}^i(ux_1)|=|\mathcal{L}^i(u)|$. Therefore, the equality (\ref{2}) is equivalent to
	\[|\mathcal{L}^i(wx_n)|=|\mathcal{L}^i(vx_n)|-|\mathcal{L}^i(u)|+1
\]
that is
	\[|\mathcal M_{d+1}|-c^{\langle d\rangle}=|\mathcal M_{d+1}|-b^{\langle d\rangle}-(|\mathcal M_d|-a)+1.
\]
Here we used the well known formula
	\[|\mathcal{M}_{d+1}\setminus\shad{\mathcal{L}}|=r^{\langle d\rangle},
\]
where $\mathcal{L}\subset\mathcal{M}_d$ is an initial lexsegment and $r=|\mathcal{M}_d\setminus\mathcal{L}|$ \cite{BH}.
Hence $I$ is Gotzmann if and only if
	\begin{eqnarray}c^{\langle d\rangle}=b^{\langle d\rangle}+{n+d-1\choose d}-a-1.\label{3}
\end{eqnarray}
By using (\ref{1}), we obtain
	\[c^{\langle d\rangle}=b^{\langle d\rangle}+c-b,
\]
that is
	\[c^{\langle d\rangle}-c=b^{\langle d\rangle}-b,
\]
which is equivalent to
\begin{eqnarray} c^{(d)}=b^{(d)}.\label{4}
\end{eqnarray}
Let us firstly consider the case $b=0$, that is $v=x_n^d$ and $I$ is the final lexsegment determined by $u$. The equation (\ref{4}) becomes
\begin{eqnarray} c^{(d)}=0.\label{5}
\end{eqnarray}
By Lemma \ref{c^d=0 dnd c<egal b}, $c^{(d)}=0$ if and only if $c\leq d$. 

For the case $b>0$, the monomial $v$ has the form 
	\[v=x_{l_1}\cdots x_{l_{d-j}}x_n^j,
\]
for some $j\geq 0$ and $1\leq l_1\leq\ldots\leq l_{d-j}\leq n-1$. The $d$-binomial expansion of $b$ is
	\[b={n-{l_1}+d-1\choose d}+\ldots+{n-l_{d-j}+j\choose j+1}.
\]
By Lemma \ref{b(d)=c(d)}, the equality (\ref{4}) holds if and only if $j\geq 1$ and $c-b\leq j$. Then we have obtained $c-b\leq j$ for any $b$. By (\ref{1}), this inequality holds if and only if ${n+d-1\choose d}-(a+1)\leq j$, that is
	\[a\geq{n+d-1\choose d}-(j+1).
\]
\end{proof}
\section{Gotzmann non-completely lexsegment ideals}

Firstly, we recall the Taylor resolution. Let $I$ be a monomial ideal of $S$ with the minimal monomial generating set $G(I)=\{u_1,\ldots, u_r\}$. The Taylor resolution $(T_{\bullet}(I),d_{\bullet})$ of $I$ is defined as follows. Let $L$ be the free $S$-module with the basis $\{e_1,\ldots,e_r\}$. Then $T_q(I)=\bigwedge\limits^{q+1}L$ for $0\leq q\leq r-1$ and $d_q:T_q(I)\rightarrow T_{q-1}(I)$ for $1\leq q\leq r-1$ is defined as follows
	\[d_q(e_{i_0}\wedge\ldots\wedge e_{i_q})=\sum_{s=0}^q(-1)^s\frac{\lcm(u_{i_0},\ldots,u_{i_q})}{\lcm(u_{i_0},\ldots ,\check{u}_{i_s},\ldots,u_{i_q})}e_{i_0}\wedge\ldots\wedge\check{e}_{i_s}\wedge\ldots\wedge e_{i_q}. 
\]
The augmentation $\varepsilon:T_0\rightarrow I$ is defined by $\varepsilon(e_i)=u_i$ for all $1\leq i\leq q$. It is known that, in general, the Taylor resolution is not minimal. M. Okudaira and Y. Takayama characterized all the monomial ideals with linear resolutions whose Taylor resolutions are minimal.
\begin{Theorem}\label{ot}\cite{OT} Let $I$ be a monomial ideal with linear resolution. The following conditions are equivalent:
\begin{itemize}
	\item[(i)] The Taylor resolution of $I$ is minimal;
	\item[(ii)] $I=m\cdot (x_{i_1},\ldots,x_{i_l})$ for some $1\leq {i_1}<\ldots<i_l\leq n$ and for a monomial $m$. 
\end{itemize}
\end{Theorem}

In \cite{HHMT}, the componentwise linear monomial ideals whose Taylor resolutions are minimal are described.
\begin{Theorem}\label{minTay}\cite{HHMT} Let $I$ be a componentwise linear monomial ideal of $S$. The following conditions are equivalent:
\begin{itemize}
	\item[(i)] The Taylor resolution of $I$ is minimal;
	\item[(ii)] $\max\{m(u):u\in G(I)\}=|G(I)|$;
	\item[(iii)] $I$ is a Gotzmann ideal with $|G(I)|\leq n$.
\end{itemize}
\end{Theorem}

Now we can complete the characterization of non-completely lexsegment ideals which are Gotzmann.

\begin{Theorem} Let $u=x_t^{a_t}\cdots x_{n}^{a_n}$, $v=x_t^{b_t}\cdots x_{n}^{b_n}$ be two monomials of degree $d$, $u>_{lex} v$, $a_t\neq0$, $t\geq1$ and $I=(\mathcal{L}(u,v))$ a non-completely lexsegment ideal. Then $I$ is a Gotzmann ideal in $S$ if and only if $I=m(x_{l},x_{l+1},\ldots,x_{l+p})$ for some $t\leq l\leq n$, some $1\leq p\leq n-l$ and a monomial $m$.
\end{Theorem}

\begin{proof} If $I=m(x_{l},x_{l+1},\ldots,x_{l+p})$ for some $t\leq l\leq n$, some $1\leq p\leq n-l$ and a monomial $m$, then $I$ is isomorphic to the monomial prime ideal $(x_{l},x_{l+1},\ldots,x_{l+p})$ and the Koszul complex of the sequence $x_{l},x_{l+1},\ldots,x_{l+p}$ is isomorphic to the minimal graded free resolution of $I$. Therefore $I$ has a linear resolution and, by Theorem \ref{ot}, the Taylor resolution of $I$ is minimal. Since any ideal with a linear resolution is componentwise linear, it follows by Theorem \ref{minTay} that $I$ is a Gotzmann ideal. 

Now it remains to prove that, if $I$ is a Gotzmann ideal in $S$, then $I$ has the required form.

Firstly, we prove that $\projdim(S/I)<n$. For this, we study the following cases.

\textit{Case I: $t=1$, $b_1=0,a_1=1$.} Since $I$ is a non-completely lexsegment ideal which is Gotzmann, $I$ has a linear resolution. Therefore, by \cite[Theorem 2.4]{ADH}, $u$ and $v$ have the form
	\[u=x_1x_{l+1}^{a_{l+1}}\ldots x_n^{a_n}\ \mbox{and}\ v=x_lx_n^{d-1}
\]
for some $l$, $2\leq l\leq n-1$. Since $x_nu<_{lex}x_1v$, using \cite[Proposition 3.2]{EOS} we get $\depth(S/I)\neq0$. Hence $\projdim(S/I)<n$.

\textit{Case II: $t=1$, $0<b_1<a_1$.}  Since $I$ is a non-completely lexsegment ideal, we must have $b_1=a_1-1$. Now, if $I$ does not have a linear resolution, $I$ is not Gotzmann. The ideal $I$ has a linear resolution if and only if $J=(\mathcal{L}(u',v'))$ has a linear resolution, where $u'=u/x_1^{b_1}$ and $v'=v/x_1^{b_1}$. One may easy check that $J$ is a non-completely lexsegment ideal. Therefore $J$ has a linear resolution if and only if $u'$ and $v'$ have the form
	$u'=x_1x_{l+1}^{a_{l+1}}\ldots x_n^{a_n}\ \mbox{and}\ v'=x_lx_n^{d-1}$
for some $l$, $2\leq l\leq n-1$ \cite[Theorem 2.4]{ADH} and this implies $u=x_1^{a_1}x_{l+1}^{a_{l+1}}\ldots x_n^{a_n}\ \mbox{and}\ v'=x_1^{a_1-1}x_lx_n^{d-1}$.
Since $x_nu<_{lex}x_1v$, using \cite[Proposition 3.2]{EOS} we get $\depth(S/I)\neq0$. Hence $\projdim(S/I)<n$. 

\textit{Case III: $t=1$, $a_1=b_1>0$.} Since $x_nu<_{lex}x_1v$, we have that $\depth(S/I)\neq0$ by \cite[Proposition 3.2]{EOS}. Hence $\projdim(S/I)<n$.

\textit{Case IV: $t>1$.} We obviously have $x_nu<_{lex}x_1v$ and, by \cite[Proposition 3.2]{EOS}, $\depth(S/I)\neq0$. Therefore $\projdim(S/I)<n$.

We may conclude that $\projdim(S/I)<n$ in all the cases.

Let $w\in\mathcal{M}_d$ be a monomial such that $|\mathcal{L}(u,v)|=|\mathcal{L}^i(w)|$. Since $I$ is a Gotzmann ideal, $I^{lex}$ is generated in degree $d$, that is $I^{lex}=(\mathcal{L}^i(w))$. By \cite[Corollary 1.4]{HH}, $I$ and $I^{lex}$ have the same Betti numbers. In particular, we have
	\[\projdim(I)=\projdim(I^{lex}).
\]
Since $\projdim(S/I)<n$, we have $\projdim(S/I^{lex})<n$. The ideal $I^{lex}$ is stable in the sense of Eliahou and Kervaire, thus there exists $j<n$ such that $w=x_1^{d-1}x_j$. Therefore, $|\mathcal{L}(u,v)|=j<n$. By the hypothesis, $I$ is a Gotzmann ideal and $I$ is componentwise linear since it has a linear resolution. By Theorem \ref{minTay}, the Taylor resolution of $I$ is minimal. The conclusion follows by Theorem \ref{ot} and taking into account that $I$ is a lexsegment ideal.
\end{proof}

\end{document}